\documentclass[12pt]{article}
\usepackage[utf8]{inputenc}
\usepackage{graphicx}
\usepackage{amsmath,amssymb}
\usepackage{amsthm}
\usepackage{soul}
\usepackage{tikz}
\usepackage{float}
\usepackage{hyperref}
\usetikzlibrary{positioning,matrix}
\usepackage{hyperref}
\usepackage{rotating}
\usepackage[margin=2cm]{geometry}

\newtheorem{theorem}{Theorem}[section]

\newtheorem{proposition}[theorem]{Proposition}

\newtheorem{example}[theorem]{Example}

\newtheorem{claim}[theorem]{Claim}

\usepackage{blkarray}
\usepackage{cleveref}

\usepackage{fancyhdr}
\pgfdeclarelayer{nice}   
\pgfsetlayers{nice,main} 
\newcommand{\MyLines}[1]{%
	\begin{tikzpicture}[scale=1.5]
		\foreach \n in {1,...,#1}
		{\pgfmathsetmacro\Division{360/#1}
			\pgfmathsetmacro\Divisions{(360*\n)/#1}
			\draw[fill=white] (90-\Divisions:3cm) coordinate (v\n) circle[radius=.5em] node {\n};}
		\begin{pgfonlayer}{nice}
			\foreach \m/\n in {1/1,1/2,1/3,1/4,1/5,1/6,1/7,1/8,1/9,1/10,1/11,1/12,1/13,1/14,1/15,1/16,1/17,1/18,1/19,1/20,1/21,1/22,1/23,1/24,1/25,1/26,1/27,1/28,1/29,1/30,1/31,1/32,1/33,1/34,1/35,1/36,2/1,2/3,2/5,2/7,2/9,2/11,2/13,2/15,2/17,2/19,2/21,2/23,2/25,2/27,2/29,2/31,2/33,2/35,3/1,3/2,3/4,3/5,3/7,3/8,3/10,3/11,3/13,3/14,3/16,3/17,3/19,3/20,3/22,3/23,3/25,3/26,3/28,3/29,3/31,3/32,3/34,3/35,4/1,4/3,4/5,4/7,4/9,4/11,4/13,4/15,4/17,4/19,4/21,4/23,4/25,4/27,4/29,4/31,4/33,4/35,5/1,5/2,5/3,5/4,5/6,5/7,5/8,5/9,5/11,5/12,5/13,5/14,5/16,5/17,5/18,5/19,5/21,5/22,5/23,5/24,5/26,5/27,5/28,5/29,5/31,5/32,5/33,5/34,5/36,6/1,6/5,6/7,6/11,6/13,6/17,6/19,6/23,6/25,6/29,6/31,6/35,7/1,7/2,7/3,7/4,7/5,7/6,7/8,7/9,7/10,7/11,7/12,7/13,7/15,7/16,7/17,7/18,7/19,7/20,7/22,7/23,7/24,7/25,7/26,7/27,7/29,7/30,7/31,7/32,7/33,7/34,7/36,8/1,8/3,8/5,8/7,8/9,8/11,8/13,8/15,8/17,8/19,8/21,8/23,8/25,8/27,8/29,8/31,8/33,8/35,9/1,9/2,9/4,9/5,9/7,9/8,9/10,9/11,9/13,9/14,9/16,9/17,9/19,9/20,9/22,9/23,9/25,9/26,9/28,9/29,9/31,9/32,9/34,9/35,10/1,10/3,10/7,10/9,10/11,10/13,10/17,10/19,10/21,10/23,10/27,10/29,10/31,10/33,11/1,11/2,11/3,11/4,11/5,11/6,11/7,11/8,11/9,11/10,11/12,11/13,11/14,11/15,11/16,11/17,11/18,11/19,11/20,11/21,11/23,11/24,11/25,11/26,11/27,11/28,11/29,11/30,11/31,11/32,11/34,11/35,11/36,12/1,12/5,12/7,12/11,12/13,12/17,12/19,12/23,12/25,12/29,12/31,12/35,13/1,13/2,13/3,13/4,13/5,13/6,13/7,13/8,13/9,13/10,13/11,13/12,13/14,13/15,13/16,13/17,13/18,13/19,13/20,13/21,13/22,13/23,13/24,13/25,13/27,13/28,13/29,13/30,13/31,13/32,13/33,13/34,13/35,13/36,14/1,14/3,14/5,14/9,14/11,14/13,14/15,14/17,14/19,14/23,14/25,14/27,14/29,14/31,14/33,15/1,15/2,15/4,15/7,15/8,15/11,15/13,15/14,15/16,15/17,15/19,15/22,15/23,15/26,15/28,15/29,15/31,15/32,15/34,16/1,16/3,16/5,16/7,16/9,16/11,16/13,16/15,16/17,16/19,16/21,16/23,16/25,16/27,16/29,16/31,16/33,16/35,17/1,17/2,17/3,17/4,17/5,17/6,17/7,17/8,17/9,17/10,17/11,17/12,17/13,17/14,17/15,17/16,17/18,17/19,17/20,17/21,17/22,17/23,17/24,17/25,17/26,17/27,17/28,17/29,17/30,17/31,17/32,17/33,17/35,17/36,18/1,18/5,18/7,18/11,18/13,18/17,18/19,18/23,18/25,18/29,18/31,18/35,19/1,19/2,19/3,19/4,19/5,19/6,19/7,19/8,19/9,19/10,19/11,19/12,19/13,19/14,19/15,19/16,19/17,19/18,19/20,19/21,19/22,19/23,19/24,19/25,19/26,19/27,19/28,19/29,19/30,19/31,19/32,19/33,19/34,19/35,19/36,20/1,20/3,20/7,20/9,20/11,20/13,20/17,20/19,20/21,20/23,20/27,20/29,20/31,20/33,21/1,21/2,21/4,21/5,21/8,21/10,21/11,21/13,21/16,21/17,21/19,21/20,21/22,21/23,21/25,21/26,21/29,21/31,21/32,21/34,22/1,22/3,22/5,22/7,22/9,22/13,22/15,22/17,22/19,22/21,22/23,22/25,22/27,22/29,22/31,22/35,23/1,23/2,23/3,23/4,23/5,23/6,23/7,23/8,23/9,23/10,23/11,23/12,23/13,23/14,23/15,23/16,23/17,23/18,23/19,23/20,23/21,23/22,23/24,23/25,23/26,23/27,23/28,23/29,23/30,23/31,23/32,23/33,23/34,23/35,23/36,24/1,24/5,24/7,24/11,24/13,24/17,24/19,24/23,24/25,24/29,24/31,24/35,25/1,25/2,25/3,25/4,25/6,25/7,25/8,25/9,25/11,25/12,25/13,25/14,25/16,25/17,25/18,25/19,25/21,25/22,25/23,25/24,25/26,25/27,25/28,25/29,25/31,25/32,25/33,25/34,25/36,26/1,26/3,26/5,26/7,26/9,26/11,26/15,26/17,26/19,26/21,26/23,26/25,26/27,26/29,26/31,26/33,26/35,27/1,27/2,27/4,27/5,27/7,27/8,27/10,27/11,27/13,27/14,27/16,27/17,27/19,27/20,27/22,27/23,27/25,27/26,27/28,27/29,27/31,27/32,27/34,27/35,28/1,28/3,28/5,28/9,28/11,28/13,28/15,28/17,28/19,28/23,28/25,28/27,28/29,28/31,28/33,29/1,29/2,29/3,29/4,29/5,29/6,29/7,29/8,29/9,29/10,29/11,29/12,29/13,29/14,29/15,29/16,29/17,29/18,29/19,29/20,29/21,29/22,29/23,29/24,29/25,29/26,29/27,29/28,29/30,29/31,29/32,29/33,29/34,29/35,29/36,30/1,30/7,30/11,30/13,30/17,30/19,30/23,30/29,30/31,31/1,31/2,31/3,31/4,31/5,31/6,31/7,31/8,31/9,31/10,31/11,31/12,31/13,31/14,31/15,31/16,31/17,31/18,31/19,31/20,31/21,31/22,31/23,31/24,31/25,31/26,31/27,31/28,31/29,31/30,31/32,31/33,31/34,31/35,31/36,32/1,32/3,32/5,32/7,32/9,32/11,32/13,32/15,32/17,32/19,32/21,32/23,32/25,32/27,32/29,32/31,32/33,32/35,33/1,33/2,33/4,33/5,33/7,33/8,33/10,33/13,33/14,33/16,33/17,33/19,33/20,33/23,33/25,33/26,33/28,33/29,33/31,33/32,33/34,33/35,34/1,34/3,34/5,34/7,34/9,34/11,34/13,34/15,34/19,34/21,34/23,34/25,34/27,34/29,34/31,34/33,34/35,35/1,35/2,35/3,35/4,35/6,35/8,35/9,35/11,35/12,35/13,35/16,35/17,35/18,35/19,35/22,35/23,35/24,35/26,35/27,35/29,35/31,35/32,35/33,35/34,35/36,36/1,36/5,36/7,36/11,36/13,36/17,36/19,36/23,36/25,36/29,36/31,36/35}
			\draw (v\n)--(v\m);
		\end{pgfonlayer}
	\end{tikzpicture}
}

\newcommand{\myLines}[1]{%
	\begin{tikzpicture}[scale=1.5]
		\foreach \n in {2,3,4,5,6,8,9,10,12,14,15,16,18,20,21,22,24,25,26,27,28,30,32,33,34,35,36}
		{\pgfmathsetmacro\Division{360/#1}
			\pgfmathsetmacro\Divisions{(360*\n)/#1}
			\draw[fill=white] (90-\Divisions:3cm) coordinate (v\n) circle[radius=.5em] node {\n};}
		\begin{pgfonlayer}{nice}
			\foreach \m/\n in {2/3,2/5,2/9,2/15,2/21,2/25,2/27,2/33,2/35,
				3/2,3/4,3/5,3/8,3/10,3/14,3/16,3/20,3/22,3/25,3/26,3/28,3/32,3/34,3/35,
				4/3,4/5,4/9,4/15,4/21,4/25,4/27,4/33,4/35,5/2,5/3,5/4,5/6,5/8,5/9,5/12,5/14,5/16,5/18,5/21,5/22,5/24,5/26,5/27,5/28,5/32,5/33,5/34,5/36,6/5,6/25,6/35,8/3,8/5,8/9,8/15,8/21,8/25,8/27,8/33,8/35,9/2,9/4,9/5,9/8,9/10,9/14,9/16,9/20,9/22,9/25,9/26,9/28,9/32,9/34,9/35,10/3,10/9,10/21,10/27,10/33,12/5,12/25,12/35,14/3,14/5,14/9,14/15,14/25,14/27,14/33,15/2,15/4,15/8,15/14,15/16,15/22,15/26,15/28,15/32,15/34,16/3,16/5,16/9,16/15,16/21,16/25,16/27,16/33,16/35,18/5,18/25,18/35,20/3,20/9,20/21,20/27,20/33,21/2,21/4,21/5,21/8,21/10,21/16,21/20,21/22,21/25,21/26,21/32,21/34,22/3,22/5,22/9,22/15,22/21,22/25,22/27,22/35,24/5,24/25,24/35,25/2,25/3,25/4,25/6,25/8,25/9,25/12,25/14,25/16,25/18,25/21,25/22,25/24,25/26,25/27,25/28,25/32,25/33,25/34,25/36,26/3,26/5,26/9,26/15,26/21,26/25,26/27,26/33,26/35,27/2,27/4,27/5,27/8,27/10,27/14,27/16,27/20,27/22,27/25,27/26,27/28,27/32,27/34,27/35,28/3,28/5,28/9,28/15,28/25,28/27,28/33,32/3,32/5,32/9,32/15,32/21,32/25,32/27,32/33,32/35,33/2,33/4,33/5,33/8,33/10,33/14,33/16,33/20,33/25,33/26,33/28,33/32,33/34,33/35,34/3,34/5,34/9,34/15,34/21,34/25,34/27,34/33,34/35,35/2,35/3,35/4,35/6,35/8,35/9,35/12,35/16,35/18,35/22,35/24,35/26,35/27,35/32,35/33,35/34,35/36,36/5,36/25,36/35}
			\draw (v\n)--(v\m);
		\end{pgfonlayer}
	\end{tikzpicture}
}

\fancypagestyle{mypagestyle}{%
	\fancyhf{}
	\fancyhead[OC]{Subarsha Banerjee}
	\fancyhead[EC]{On Structural Properties and Adjacency Spectrum of Coprime Graph of Integers}
	\fancyfoot[C]{\thepage}%
}
\title{On Structural Properties and Adjacency Spectrum of Coprime Graph of Integers}
\author{Subarsha Banerjee
	\\	Department of  Mathematics,  JIS University 
	\\
	81 Nilgunj Road, Agarpara, West Bengal 700109\\
	e-mail:subarshabnrj@gmail.com/subarsha.banerjee@jisuniversity.ac.in
}
\date{}

\begin{document}
	\maketitle
	
	\begin{abstract}
		Let $TCG_n$ denote the coprime graph having vertex set $\{1,2,\ldots,n\}$ with any two vertices $i,j$ being adjacent if and only if $\gcd(i,j)=1$.
		In this article, we first study some structural properties of $TCG_n$.
		We study the vertex connectivity and crossing number of the coprime graph of integers. We discover a lower constraint on the multiplicity of $-1$, which appears as an eigenvalue in the adjacency matrix of $TCG_n$.
		We demonstrate our findings with a variety of cases.  We also show that the adjacency matrix of $TCG_n$ is singular, i.e. has determinant $0$. Furthermore, we give a lower bound on the multiplicity of   $0$, which appears as an eigenvalue in the adjacency matrix of $TCG_n$.
		Finally, we establish that the greatest eigenvalue of the adjacency matrix of $TCG_n$ is always above $2.$
	\end{abstract}
	
	\section{Introduction}
	
	Let $TCG_n$ be the coprime graph having vertex set $\{1,2,\ldots,n\}$, where any two vertices  $i,j$ are adjacent if and only if $\gcd(i,j)=1$.
	Coprime graphs, presented by Erdös \cite{erdos1961remarks} in 1962, have sparked attention and research into various critical subjects. Erdös published a conjecture concerning subgraphs in coprime graphs in \cite{erdos1961remarks}, which sparked academic interest. Ahlswede and Khachatrian's 1995 paper \cite{ahlswede1994extremal} responded negatively to the supposition. Ahlswede and Khachatrian conducted more research on this topic \cite{ahlswede1995maximal,ahlswede1996sets}. Newman suggested a conjecture about the criteria for perfect matching in coprime graphs, later proven by Pomerance and Selfridge in \cite{pomerance1980proof}. There are also other studies. Ahlswede and Blinovsky \cite{ahlswede2006maximal} examined extremal sets without coprime elements and sets of integers with pairwise common divisors. Erdös and Sarkozy \cite{erdHos1997cycles} researched the cycle of coprime graphs, while Sander and Sander \cite{sander2009kernel} investigated the kernel of coprime graphs. A survey \cite{rao2011creative} provides additional information on the research status of coprime graphs.
	In \cite{pan2019full}, the authors determined the automorphism group, as well as the determining and resolving sets of coprime graphs.
	The concept of coprime graphs was generalized in \cite{sriram2014} and further studied in \cite{sriramgeneralised}. 
	The coprime graph of finite groups has also been an active research area. Many researchers have studied the coprime graphs of various finite groups recently, for example, \cite{dorbidi2016note,selvakumar2017classification,hamm2021parameters,banerjee2021laplacian,madhumitha2024graphs,ma2014coprime}.
	Inspired by the previously described studies, we investigate several structural properties of the coprime graph of integers.
	Studying spectral properties of graphs associated with various algebraic structures has been an active topic of research, see \cite{rehman2024exploring,rehman2024randic,rather2023normalized,shen2023laplacian,alsaluli2024laplacian,banerjee2022spectra,banerjee2022laplacian,banerjee2023distance,banerjee2023structural}.
	Consequently, we also shed light on the adjacency eigenvalues of the coprime graph of integers.
	\Cref{PL} is devoted to the basic concepts required for the sections that follow. We explore various structural characteristics of \( TCG_n \) in \Cref{S1}. Lastly, we discuss the adjacency spectrum of \( TCG_n \) in \Cref{S2}. An appendix is included in \Cref{appendix}, listing the adjacency spectra of \( TCG_n \) for \( 3 \leq n \leq 15 \), computed using \textsc{Matlab}.

	\section{Preliminaries}
	\label{PL}
	In this section, we present some preliminary theorems and definitions that we have utilized throughout the article.
	
	A graph $G$ is represented by $G=(V,E)$ where $V$ denotes the vertex set of $G$ and $E$ denotes the edge set of $G$. 
	Two vertices are considered to be \textit{adjacent} if there is an edge between them.
	A \textit{complete graph} with $n$ vertices, denoted by $K_n$ is a graph in which each pair of distinct vertices are adjacent.
	A \textit{path} $P$ of length $k$ in a graph $G$ is an alternating sequence of vertices and edges $v_0, e_0, v_1, e_1, v_2, e_2, \ldots, v_{k-1}, e_{k-1}, v_k$, where $v_i's$ are unique vertices, and $e_i$ is the edge joining $v_i$ and $v_{i+1}$. 
	If $v_0=v_k$, $P$ is considered a \textit{cycle} with length $k$.
	The \textit{girth} refers to the length of the shortest cycle in $G$.
	A graph $G$ is considered \textit{connected} if there is a path connecting any two vertices $u,v\in V$.
	The term \textit{disconnected} graph refers to a graph that is not connected.
	In a connected graph $G$, the \textit{distance} between two vertices $u,v$, indicated by $d(u,v)$, is the length of the shortest path between $u$ and $v$.
	The diameter of a connected graph $G$ is defined as diam$(G)=\max\{ d(u,v):u,v\in V\}$.
	A \textit{planar} graph can be immersed in the plane, meaning its edges intersect only at their endpoints.
	The \textit{crossing number} of a graph $G$ denoted by $\text{cr}(G)$ is the minimal number of edge crossings in a planar drawing of $G$.
	If the graph is planar, then its crossing number is $0$.
	The crossing number of the complete graph $K_n$ and complete bipartite graph $K_{m,n}$ are still not known for all $m,n$. 
	Several problems on crossing number of a graph have been discussed in \cite{erdos1973crossing}. Finding the crossing number of a graph is an NP-complete problem \cite{garey1983crossing}.
	Readers may refer to \cite{schaefer2012graph} for a survey on the crossing number of a graph.
	The \textit{vertex connectivity} $\kappa(G)$ of a graph $G$ is the least number of vertices that may be removed to create a disconnected graph. We define a disconnected graph's connectivity as $0$.
	A \textit{triangulated} graph has an edge connecting two non-adjacent vertices in each cycle of length four or more.
	A graph's \textit{clique number} refers to the size of its greatest complete subgraph.
	A graph is said to be \textit{bipartite} if its vertex set can be partitioned into two disjoint sets $P$ and $Q$ such that any edge in the graph will have one of its end-points in $A$, and the other in $B$.
	A \textit{complete bipartite graph} denoted by $K_{m,n}$ is a special kind of bipartite graph in which there exists an edge between every vertex of $P$ and every vertex of $Q$.
	The \textit{adjacency matrix} of a graph having vertices $v_1, v_2, \dots, v_n$ is an $n\times n$ matrix whose $(i,j)^{\text{th}}$ entry is $1$ if there is an edge between $v_i$ and $v_j$ and it is $0$ otherwise.
	The adjacency matrix of $G$ is denoted by $A(G)$, and it is a symmetric matrix. Consequently, the eigenvalues of $A(G)$ are real numbers. Let the eigenvalues of $A(G)$ be denoted by $\lambda_1\le \lambda_2\le \lambda_3\le \cdots\le \lambda_{n-1}\le \lambda_n$.
	The largest eigenvalue of $A(G)$, i.e. $\lambda_n$ is known as the \textit{spectral radius} of $G$.
	The remaining results in the paper depend on the following theorems.
	The remaining results in the paper depend on the following theorems.
	
	\begin{theorem}[Kuratowski's Theorem {\cite[Theorem~6.2.2]{west2001introduction}}] \label{Kura}
		A graph is planar if and only if it does not contain a subdivision of \( K_5 \) or \( K_{3,3} \).
	\end{theorem}
	
	\begin{theorem}[Eigenvalue Interlacing Theorem {\cite{brouwer2011spectra}}] \label{ELT}
		Suppose \( A \in \mathbb{R}^{n \times n} \) is a symmetric matrix, and let \( B \in \mathbb{R}^{m \times m} \) with \( m < n \) be a principal submatrix of \( A \). Suppose the eigenvalues of \( A \) are 
		\[
		\lambda_1 \le \lambda_2 \le \cdots \le \lambda_n,
		\]
		and the eigenvalues of \( B \) are
		\[
		\beta_1 \le \beta_2 \le \cdots \le \beta_m.
		\]
		Then, for each \( k = 1, 2, \ldots, m \), we have
		\[
		\lambda_k \le \beta_k \le \lambda_{k + n - m}.
		\]
		Moreover, if \( m = n - 1 \), then the interlacing becomes
		\[
		\lambda_1 \le \beta_1 \le \lambda_2 \le \beta_2 \le \cdots \le \beta_{n-1} \le \lambda_n.
		\]
	\end{theorem}

	\begin{theorem}[Bertrand's Postulate {\cite{ireland2013classical}}] \label{BT}
		For every natural number \( n \), there always exists at least one prime number between \( n \) and \( 2n \).
	\end{theorem}
	
	\begin{theorem}[Characterization of Bipartite Graphs {\cite[Theorem~1.2.18]{west2001introduction}}] \label{CY}
		A graph is bipartite if and only if it does not contain any odd cycle.
	\end{theorem}

	\section{Structural Properties of $TCG_n$}
	\label{S1}
	\begin{proposition}
		If $n>3$, then	$TCG_n$ is connected and has diameter $2$.
	\end{proposition}
	\begin{proof}
		Consider the vertex $1\in TCG_n$.
		For any $v(\neq 1)\in TCG_n$, $\gcd(v,1)=1$ which implies that $v$ is adjacent to $1$.
		Hence \textit{distance} between any two vertices of $TCG_n$ is at most $2$.
		Thus \textit{diameter} of $TCG_n$ equals $2$.
	\end{proof}

	\begin{proposition}
		$TCG_n$ is complete if and only if $n=2 \text{ or } 3$.
	\end{proposition}

	\begin{proof}
		If $n=2$,  then the vertices of $TCG_n$ are $\{1,2\}$ which is complete. If $n=3$, then the vertices of $TCG_n$ are $\{1,2,3\}$, which can be readily seen to be complete.
		Conversely, if $n\ge 4$ then the vertices $2$ and $4$ of $TCG_n$ are not adjacent since $\gcd(2,4)=2(\neq 1)$.
		Hence $TCG_n$ is not complete for $n\ge 4$.
		\\
		Thus the result follows.
	\end{proof}
	
	\begin{proposition}
		If $n\ge 3$, then girth of $TCG_n$ is $3$ .
	\end{proposition}
	
	\begin{proof}
		Consider the vertices $\{1,2,3\}$ of $TCG_n$.
		Since the given vertices form a $3$ cycle and it is the shortest cycle $TCG_n$ possesses for any $n\ge 3$, the girth of  $TCG_n$ equals $3$. 
	\end{proof}
	
	\begin{proposition}
		If $n\ge 3$, then $TCG_n$ is not bipartite.
	\end{proposition}

	\begin{proof}
		If $n\ge 3$, then $TCG_n$ contains an odd cycle of length $3$.
		Using \Cref{CY}, we conclude that $TCG_n$ is not bipartite.
	\end{proof}

	\begin{theorem}
		If $n\ge 3$ then $TCG_n$ is triangulated.
	\end{theorem}
	
	\begin{proof}
		To show that $TCG_n$ is triangulated, we show that any vertex of $TCG_n$ is a vertex of a triangle.
		We notice that for any $m$ where $2\le m\le n$, $\gcd(m-1,m)=1$.
		Thus the vertex $m$ of $TCG_n$ is a vertex of the triangle $\Delta$ where $\Delta=\{1,m-1,m\}$. 
		Since the above holds for any $m$ where $2\le m\le n$, we conclude that $TCG_n$ is triangulated.
	\end{proof}

	\begin{theorem}
		\label{clique}
		The clique number of $TCG_n$  is $\pi(n)+1$, where $\pi(n)$ is the prime counting function.
	\end{theorem}

	\begin{proof}
		Given $n\ge 3$, consider the set $$\mathcal{S}=\{1\le x\le n: x \text{ is a prime} \}\cup \{1\}.$$
		Clearly for any $i,j\in \mathcal{S}$,  $\gcd(i,j)=1$ which makes us  conclude that  the \textit{induced} subgraph of $\mathcal{S}$ in $TCG_n$ is complete.
		\\
		We now prove that the induced subgraph of $\mathcal{S}$ in $TCG_n$ is \textit{maximal}.
		We take an element $w$ such that $w\notin \mathcal{S}$.
		We now consider  the \textit{induced} subgraph of $\mathcal{S}\cup \{w\}$ in $TCG_n$.
		Since $w\notin \mathcal{S}$, so $w$ is composite and hence it must have a prime factor say $p$.
		Now since $\gcd(p,w)=p(\neq 1)$, we find that $p$ and $w$ are not adjacent and hence the  \textit{induced} subgraph of $\mathcal{S}\cup \{w\}$ in $TCG_n$ is not complete.
		Thus the \textit{induced} subgraph of $\mathcal{S}$ in $TCG_n$ forms a \textit{maximal clique}.
		
		We find that the set $\mathcal{S}$ contains $\pi(n)+1$ elements, where  $\pi(n)$  denotes the number of \textit{primes} which are less than or equal to $n$.
		Since the number of elements in $\mathcal{S}$ is $\pi(n)+1$,  we conclude that $\omega(TCG_n)=\pi(n)+1$.

	\end{proof}
	
	\begin{theorem}
		If $n\ge 7$, then $TCG_n$ is not planar.
	\end{theorem}
	
	\begin{proof}
		If $n=7$, then using \Cref{clique} we find that  $TCG_n$ has a \textit{clique} of order $\pi(7)+1=5$.
		Thus for $n=7$, $TCG_n$ contains the complete graph $K_5$ as a subgraph.
		Using \Cref{Kura} we observe that $TCG_7$ is not planar.
		Since $TCG_n$ is a subgraph of $TCG_{n+1}$ and $TCG_n$ is not planar for $n=7$, we conclude that $TCG_n$ is not planar for $n\ge 7$.
	\end{proof}

	\begin{theorem}
		The crossing number of $TCG_n$ is at least $3$, where equality holds if and only if $n=7$.
	\end{theorem}
	
	\begin{proof}
		If $n=7$, then $TCG_n$ is of the following form:
		\begin{figure}[H]
			\centering
			\begin{tikzpicture}
				\foreach \p [evaluate =\p as \angle using 360*\p/7+50] in {1,...,7}
				\node[shape=circle,draw=black] (\p) at (\angle:3) {$\p$};  
				\foreach \x [remember= \x as \lastx (initially 7)]in{1,...,7} \draw (\lastx)--(\x);  
				
				\draw(4)--(7)--(2)--(5)--(7)--(3)--(5);
				
				\draw    (1) to[out=170,in=140,looseness=1.2] (3);
				\draw    (1) to[out=160,in=180,looseness=2] (4);
				\draw    (1) to[in=30,out=30,looseness=1.7] (5);
				\draw    (1) to[in=80,out=20,looseness=1.3] (6);       
			\end{tikzpicture} 
			\caption{$TCG_7$}
			\label{Fig1}
			
		\end{figure}
		
		Using \Cref{Fig1}, we observe that $TCG_7$ has $3$ crossings. 
		This implies that the crossing number of $TCG_7$ is at most $3$.
		Now, we claim that the crossing number of $TCG_7$ cannot be $1$. We establish our claim below.
		\begin{claim}
			\label{crossing}
			$\text{cr}(TCG_7)\neq 1$.
		\end{claim}
		\begin{proof}
			Assume that $\text{cr}(TCG_7)=1$.
			Now $TCG_7$ has $7$ vertices and $17$ edges. Since $\text{cr}(TCG_7)=1$, if we remove $1$ edge from $TCG_7$, the resultant graph containing  $7$ vertices and $16$ edges will be planar.
			Now we know that a planar graph must satisfy the following conditions:
			\begin{equation}
				\label{planar}
				e\le 3n-6, \text{where } e \text { is the number of edges and } n \text{ is the number of vertices}.
			\end{equation}
			In this case, $n=7$ and $e=16$.
			Using \Cref{planar}, $16\le (3\times 7)-6=15$ which is false.
			Thus, we arrive at a contradiction.
			Hence, our assumption that $\text{cr}(TCG_7)=1$ is false.
			This proves our claim.
		\end{proof}
		
		Now we show that $TCG_7$ can't have a crossing number $2$ as well. Let $G_1$ be the subgraph of $TCG_7$ induced by the set $\{1,2,3,5,7\}$ and let $G_2$ be the subgraph of $TCG_7$ induced by the set $\{1,3,4,5,7\}$.
		Both the graphs $G_1$ and $G_2$ are complete and hence are isomorphic to $K_5$.
		Now, we know that the crossing number of $K_5$ is $1$.
		We observe that $TCG_7$ contains subgraphs isomorphic to $K_5$ but none of them has the vertices $2$ and $4$ at the same time.
		This shows that $TCG_7$ has $2$ different crossings. Moreover, let $G_3$ the subgraph of $TCG_7$ induced by the set $\{2,6,4,1,5,7\}$. Clearly $G_3$ is isomorphic to $K_{3,3}$. Now, it is known that the crossing number of $K_{3,3}$ is $1$. Moreover, the subgraph $G_3$ also has a pair of crossing edges, and each edge of it has exactly one even vertex. Therefore, a crossing of these edges is different from the two crossings found in $G_1$ and $G_2$.
		Thus $TCG_7$ can't have a crossing number $2$. 
		Using the above arguments and  Claim \ref{crossing}, it is proved that the crossing number of $TCG_7$ is $3$. 
		Now, for any $n>7$, $TCG_n$ contains $TCG_7$ as a subgraph. Since the crossing number of $TCG_7$ is $3$, it is established that the crossing number of $TCG_n$ is greater than $3$ for all $n>7.$
		
	\end{proof}
	
	\begin{theorem}
		
		Let $n=p_1^{\alpha_1}p_2^{\alpha_2}\cdots p_m^{\alpha_m}$ where $p_i$ is  a prime and $\alpha_i$ is a positive integer for $i=1,2, \ldots, m$. Then the vertex connectivity of $TCG_n$ is given by
		$$\kappa(TCG_n)\le \text{Number of elements in } \{1\le x\le n:\gcd(x,p_k\#)=1\}$$ where $p_k\#$ is the $k^{th}$ primorial such that $p_k\#\le n$.
		
	\end{theorem}
	
	\begin{proof}
		
		Consider the largest positive integer $k$ such that $p_k\#\le n$.
		Here $p_k\# $ is defined as the product of the first $k$ primes, i.e. $p_k\#=2\times 3 \times 5\times  \dots \times p_k$ where $p_k$ is the $k^{th}$ prime.
		\\
		Let us denote $m=p_k\#$ such that $m$ is the largest \textit{primorial} not exceeding $n$. 
		We consider the vertex $m$ of $TCG_n$.
		We notice that if we remove those vertices $x$ of $TCG_n$ for which  $\gcd(m,x)=1$, then the vertex $m$ gets \textit{isolated} and $TCG_n$ becomes disconnected.
		Thus the set $\{1\le x\le n:\gcd(x,m)=1\}$  is a \textit{separating} set of $TCG_n$ and hence 
		
		\begin{align*}
			\begin{split}
				\kappa(TCG_n)\le \text{Number of elements in } \{1\le x\le n:\gcd(x,m)=1\}
				\\
				= \text{Number of elements in } \{1\le x\le n:\gcd(x,p_k\#)=1\}
			\end{split}
		\end{align*}

		where $p_k\#$ is the $k^{th}$ primorial such that $p_k\#\le n$.

	\end{proof}

	\begin{example}
		Consider $TCG_n$ for $n=36=2^2\times 3^2.$
		We consider the \textit{primorial} $p_k\#$ where $k$ is the largest integer such that $p_k\#\le 36$.
		Here $k=3$ and $p_3\#=2\times 3\times 5=30$.
		Now $TCG_{36}$ has the following form:
		If we remove the vertices which are coprime to $30$, i.e. if we remove the vertices $\{1,7,11,13,17,19,23,29,31\}$ from $TCG_{36}$ then the vertex $30$ becomes isolated which in turn makes $TCG_{36}$ disconnected.
		Thus $\{1,7,11,13,17,19,23,29,31\}$ is a separating set of $TCG_{36}$ and  $\kappa(TCG_{36})\leq 9$.
		
		
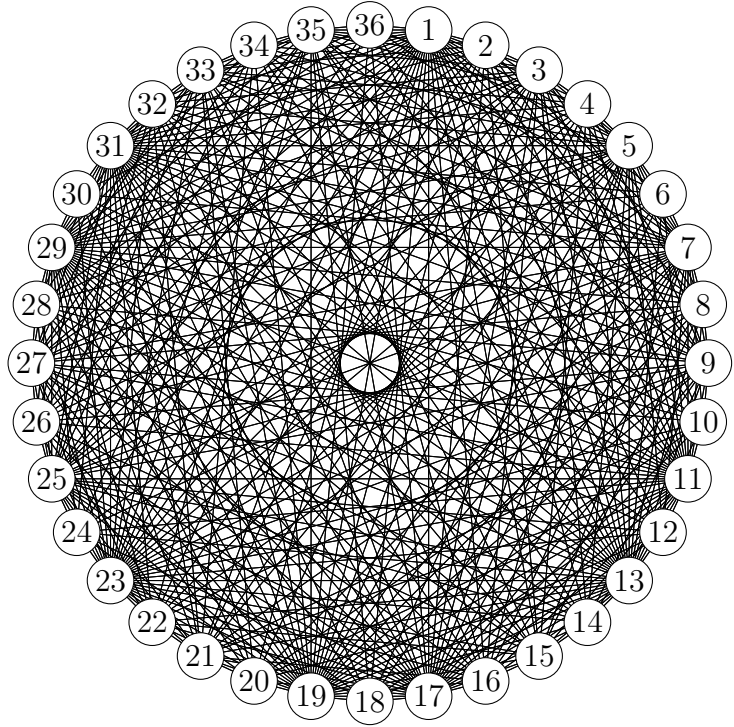
\begin{figure}[H]
			\centering
			
			\MyLines{36}
		\end{figure}

		
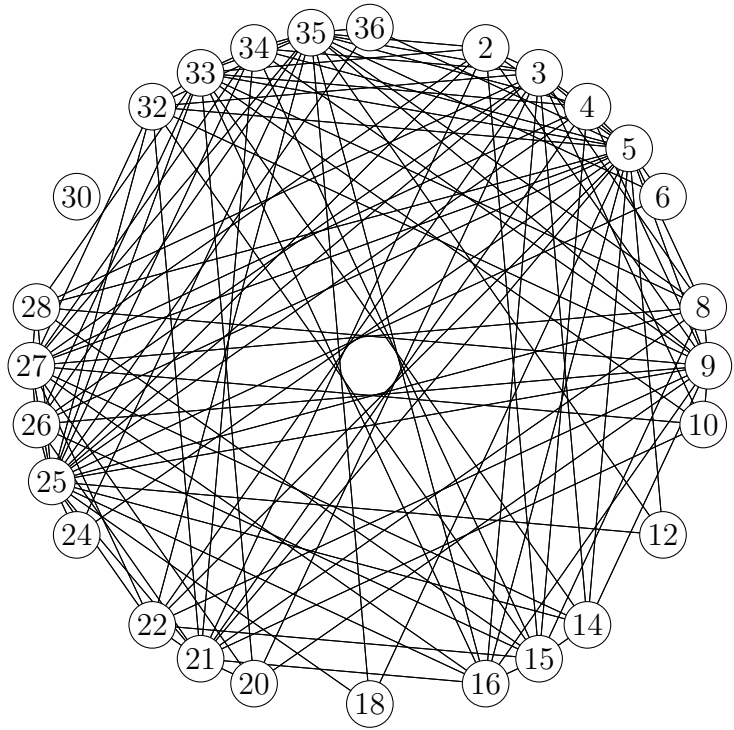
\begin{figure}[H]
			\centering
			\myLines{36}
		\end{figure}
		
	\end{example}
	
	\section{Adjacency Spectrum of Coprime Graph of Integers}
	\label{S2}
	In this section, we shall study the adjacency spectrum of the coprime graph of integers. We find a lower bound on the multiplicity of $-1$ which occurs as an eigenvalue of the adjacency matrix of $TCG_n$.
	We illustrate our results using various examples.  We further prove that the adjacency matrix of $TCG_n$ is singular, i.e. it has determinant $0$. Moreover, we also provide a lower bound on the multiplicity of the eigenvalue $0$ which appears as an eigenvalue of the adjacency matrix of $TCG_n$.
	In the end, we prove that the largest eigenvalue of the adjacency matrix of $TCG_n$ is always greater than $2.$

	\begin{theorem}
		\label{number}
		The number of times $-1$ appears as an eigenvalue of the adjacency matrix of $TCG_n$ is at least equal to the cardinality of $\mathcal{D}$, where $\mathcal{D}$ is the set of prime numbers which are greater than $\frac{n}{2}$ but less than $n$. 
	\end{theorem}
	
	\begin{proof}
		Let $\mathcal{D}$ be the collection of all the prime numbers which are greater than $\frac{n}{2}$ but less than $n$. Let $p\in \mathcal{D}$. 
		Then $\gcd(p,a)=1$ for all $1\le a< p$ and $p<a\le n$.
		This implies that $p$ is adjacent to all the members of the graph $TCG_n$.
		Thus, the subgraph of $TCG_n$ induced by the set $\mathcal{D}\cup \{1\}$ is a complete graph on $|\mathcal{D}|+1$ vertices.
		We shall index the rows and columns of $A(TCG_n)$ as follows:
		We first index the vertex $1$, we then index all the vertices of $TCG_n$ which are present in $\mathcal{D}$, and finally we index all the remaining vertices of 
		$TCG_n$.
		Using the above indexing, we find that $A(TCG_n)$ takes the following form:
		
		\begin{equation}
			\label{Eq0}
			\begin{split}
				A=A(TCG_n)&=
				\\
				&
				\left[
				\begin{array}{cccccccccccccccccc}
					0 & 1 & 1& \ldots &\dots & \textbf{1}&1 & 1 & 1& \ldots&\ldots &\ldots& 1\\
					1 & 0 & 1& \ldots &\dots &\textbf{1}&1 & 1 & 1& \ldots &\ldots&\ldots& 1\\
					1& 1 & 0 &\ldots&\dots& \textbf{1}&1 & 1 & 1& \ldots &\ldots&\ldots& 1\\
					\ldots& \ldots &\ldots &\ldots&\ldots& \ldots &\ldots &\ldots&\ldots&\ldots&\ldots&\ldots&\ldots\\
					\ldots& \ldots &\ldots &\ldots&\ldots& \ldots &\ldots &\ldots&\ldots&\ldots&\ldots&\ldots&\ldots\\
					\ldots& \ldots &\ldots &\ldots&\ldots& \ldots &\ldots &\ldots&\ldots&\ldots&\ldots&\ldots&\ldots\\
					\textbf{1}& \textbf{1} &\textbf{1}&\ldots &\ldots& \textbf{0} &\textbf{1} & \textbf{1} & \textbf{1}& \ldots &\ldots&\ldots& \textbf{1}\\
					1 & 1 &1  & \ldots&\ldots&\textbf{1} & 0 & a_{12} &a_{13}& \ldots&\ldots&\ldots&a_{1m}\\
					1 & 1 &1 & \ldots&\ldots & \textbf{1}& a_{21} & 0 &a_{23}& \ldots&\ldots&\ldots&a_{2m}\\
					1& 1 & 1 &\ldots&\dots& \textbf{1}&a_{31} & a_{32} & 0& \ldots &\ldots&\ldots& a_{3m}\\
					\ldots& \ldots &\ldots &\ldots&\ldots& \ldots &\ldots &\ldots&\ldots& \ldots&\ldots&\ldots&\ldots \\
					\ldots& \ldots &\ldots &\ldots&\ldots& \ldots &\ldots &\ldots&\ldots& \ldots &\ldots&\ldots&\ldots\\
					\ldots& \ldots &\ldots &\ldots&\ldots& \ldots &\ldots &\ldots&\ldots& \ldots &\ldots&\ldots&\ldots\\
					\ldots& \ldots &\ldots &\ldots&\ldots& \ldots &\ldots &\ldots&\ldots& \ldots &\ldots&\ldots&\ldots\\
					\ldots& \ldots &\ldots &\ldots&\ldots& \ldots &\ldots &\ldots&\ldots&\ldots&\ldots&\ldots&\ldots\\
					1 & 1 & 1 &  \ldots& \dots&\textbf{1}& a_{m1} & a_{m2} &  a_{m3}&\ldots&\ldots&\ldots& 0\\
				\end{array}
				\right].
			\end{split}
		\end{equation}
		Here $m=n-1-|\mathcal{D}|.$

		Consider the matrix $A+I$, where $I$ is the identity matrix of order $n$.
		We have,
		
		\begin{equation}
			\label{Eq1}
			\begin{split}
				A+I&=
				\\
				&
				\left[
				\begin{array}{cccccccccccccccccc}
					1 & 1 & 1& \ldots &\dots & \textbf{1}&1 & 1 & 1& \ldots&\ldots &\ldots& 1\\
					1 & 1 & 1& \ldots &\dots &\textbf{1}&1 & 1 & 1& \ldots &\ldots&\ldots& 1\\
					1& 1 & 1 &\ldots&\dots& \textbf{1}&1 & 1 & 1& \ldots &\ldots&\ldots& 1\\
					\ldots& \ldots &\ldots &\ldots&\ldots& \ldots &\ldots &\ldots&\ldots&\ldots&\ldots&\ldots&\ldots\\
					\ldots& \ldots &\ldots &\ldots&\ldots& \ldots &\ldots &\ldots&\ldots&\ldots&\ldots&\ldots&\ldots\\
					\ldots& \ldots &\ldots &\ldots&\ldots& \ldots &\ldots &\ldots&\ldots&\ldots&\ldots&\ldots&\ldots\\
					\textbf{1}& \textbf{1} &\textbf{1}&\ldots &\ldots& \textbf{1} &\textbf{1} & \textbf{1} & \textbf{1}& \ldots &\ldots&\ldots& \textbf{1}\\
					1 & 1 &1  & \ldots&\ldots&\textbf{1} & 1 & a_{12} &a_{13}& \ldots&\ldots&\ldots&a_{1m}\\
					1 & 1 &1 & \ldots&\ldots & \textbf{1}& a_{21} & 1 &a_{23}& \ldots&\ldots&\ldots&a_{2m}\\
					1& 1 & 1 &\ldots&\dots& \textbf{1}&a_{31} & a_{32} & 1& \ldots &\ldots&\ldots& a_{3m}\\
					\ldots& \ldots &\ldots &\ldots&\ldots& \ldots &\ldots &\ldots&\ldots& \ldots&\ldots&\ldots&\ldots \\
					\ldots& \ldots &\ldots &\ldots&\ldots& \ldots &\ldots &\ldots&\ldots& \ldots &\ldots&\ldots&\ldots\\
					\ldots& \ldots &\ldots &\ldots&\ldots& \ldots &\ldots &\ldots&\ldots& \ldots &\ldots&\ldots&\ldots\\
					\ldots& \ldots &\ldots &\ldots&\ldots& \ldots &\ldots &\ldots&\ldots& \ldots &\ldots&\ldots&\ldots\\
					\ldots& \ldots &\ldots &\ldots&\ldots& \ldots &\ldots &\ldots&\ldots&\ldots&\ldots&\ldots&\ldots\\
					1 & 1 & 1 &  \ldots& \dots&\textbf{1}& a_{m1} & a_{m2} &  a_{m3}&\ldots&\ldots&\ldots& 1\\
				\end{array}
				\right].
			\end{split}
		\end{equation}
		
		We find that the first $\mathcal{D}+1$ rows of the matrix $A+I$ are identical.
		Hence we apply the row operation $R_i=R_i-R_1$ for $2\le i\le \mathcal{D}+1$.
		On applying the given row operations, the matrix $A+I$ reduces to the following matrix:
		
		\begin{equation}
			\label{Eq2}
			\left[
			\begin{array}{cccccccccccccccccccccccccccccccccccccccccccccc}
				1 & 1 & 1& \ldots &\dots & \textbf{1}&1 & 1 & 1& \ldots&\ldots &\ldots& 1\\
				0 & 0 & 0& \ldots &\dots &\textbf{0}&0 & 0 & 0& \ldots &\ldots&\ldots& 0\\
				0& 0 & 0 &\ldots&\dots& \textbf{0}&0 & 0 & 0& \ldots &\ldots&\ldots& 0\\
				\ldots& \ldots &\ldots &\ldots&\ldots& \ldots &\ldots &\ldots&\ldots&\ldots&\ldots&\ldots&\ldots\\
				\ldots& \ldots &\ldots &\ldots&\ldots& \ldots &\ldots &\ldots&\ldots&\ldots&\ldots&\ldots&\ldots\\
				\ldots& \ldots &\ldots &\ldots&\ldots& \ldots &\ldots &\ldots&\ldots&\ldots&\ldots&\ldots&\ldots\\
				\textbf{0}& \textbf{0} &\textbf{0}&\ldots &\ldots& \textbf{0} &\textbf{0} & \textbf{0} & \textbf{0}& \ldots &\ldots&\ldots& \textbf{0}\\
				\cline{1-13}
				&   &                      &   & \multicolumn{1}{c|}{}  &   &                        \\
				&   & \makebox[0pt]{\huge{$J^T$}}    &   & \multicolumn{1}{c|}{}  &&   &&\makebox[0pt]{\huge{$A(G')$}}    \\
				&   &                      &   & \multicolumn{1}{c|}{}  &   &
			\end{array} \;
			\right].
		\end{equation}
		Here $G'$ is the induced subgraph of $TCG_n$ on the set $\mathbb Z_n\setminus \{\mathcal{D}\cup \{1\}\}.$ 
		$A(G')$ denotes the adjacency matrix of $G'$, and $J$ denotes the all one matrix.
		We note that the matrix in \Cref{Eq2}  has $\mathcal{D}$ identical rows of the form: 
		$\begin{array}{ccccccccccccccccc}
			(0 & 0 & 0 &\dots & 0 &0 &0&0&0&\dots&0 &0 ).
		\end{array}$
		Thus rank of $A+I$ is less than $n-\mathcal{D}$. Using the Rank-Nullity Theorem, we find that  $0$ is an eigenvalue of $A+I$ with multiplicity at least $\mathcal{D}$. Hence, we conclude that $-1$ is an eigenvalue of $A$ with multiplicity at least $\mathcal{D}$.

	\end{proof}

	\begin{proposition}
		If $n=3$, then $-1$ is an eigenvalue of the adjacency matrix of $TCG_n$ with multiplicity $2$.
	\end{proposition}
	
	\begin{proof}
		If $n=3$, then $TCG_n$ is a complete graph on $3$ vertices. Hence the result follows.
	\end{proof}
	
	\begin{theorem}
		\label{number1}
		If $n>3$, then $-1$ is an eigenvalue of the adjacency matrix of $TCG_n$ with multiplicity at least $1$.
	\end{theorem}
	
	\begin{proof}
		Consider $n>3$. Then using \Cref{BT} we can always find a prime number greater than $\frac{n}{2}$ and less than $n$. Thus $\mathcal{D}$ is non-empty and has at least one element in it. Using \Cref{number}, we conclude that $-1$ is an eigenvalue of the adjacency matrix of $TCG_n$ with multiplicity at least $1$.
	\end{proof}
	
	\begin{theorem}
		\label{number2}
		If $n>3$ is a prime number, then $-1$ is an eigenvalue of the adjacency matrix of $TCG_n$ with multiplicity at least $2$.
	\end{theorem}
	
	\begin{proof}
		Consider $n>3$. Then using \Cref{BT}, we can always find a prime number greater than $\frac{n}{2}$ and less than $n$. Moreover, $n$ is itself a prime number that is greater than $n$ but less than or equal to $n$. Thus $n\in \mathcal{D}$, which implies that  $\mathcal{D}$ has cardinality at least $2$. Using \Cref{number}, we conclude that $-1$ is an eigenvalue of the adjacency matrix of $TCG_n$ with multiplicity at least $2$.
	\end{proof}
	
	\begin{example}
		Suppose $n=6$.
		Then the only prime number that is less than or equal to $6$ but greater than $3$ is $5$. Then $\mathcal{D}=\{5\}$.
		We follow the same indexing as given in \Cref{number}. 
		We first index the element $1$, then the element in $\mathcal{D}$ i.e. $5$, and finally all the remaining elements of $TCG_6$.
		Using the above indexing, the adjacency matrix of $TCG_n$ takes the following form:
		\[
		\begin{blockarray}{ccccccc}
			&1 & 5 & 2 & 3 & 4 & 6 \\
			\begin{block}{c(cccccc)}
				1 & 0 & 1 & 1 & 1 & 1 & 1\\
				5 & 1 & 0 & 1 & 1 & 1 & 1 \\
				2 & 1 & 1 & 0 & 1 & 0 & 0 \\
				3 & 1 & 1 & 1 & 0 & 1 & 0\\
				4 & 1 & 1 & 1 & 1 & 0 & 0\\
				6 & 1 & 1 & 0 & 0 & 0 & 0\\
			\end{block}
		\end{blockarray}
		\]
		The eigenvalues of the adjacency matrix of $TCG_6$ are $0, -1, -2, -1.29, 0.39, 3.89$. We find that $-1$ occurs as an eigenvalue of the adjacency matrix of $TCG_6$ which is in accordance with \Cref{number1}.
	\end{example}
	
	\begin{example}
		We now consider a prime number.
		Suppose $n=11$.
		Then the only prime numbers that are less than or equal to $11$ but greater than $5.5$ are $5$. Then $\mathcal{D}=\{7,11\}$.
		We follow the same indexing as given in \Cref{number}. 
		We first index the element $1$, then the elements in $\mathcal{D}$ i.e. $7$ and $11$, and finally all the remaining elements of $TCG_6$.
		Using the above indexing, the adjacency matrix of $TCG_n$ takes the following form:
		\[
		\begin{blockarray}{ccccccccccccccccccccccccc}
			&1 & 7 & 11 & 2 & 3 & 4 & 5 & 6 & 8 & 9 & 10 \\
			\begin{block}{c(cccccccccccccccccccccccc)}
				1 & 0&1&1&1&1&1&1&1&1&1&1\\
				7 & 1&0&1&0&1&0&1&0&1&0&1\\
				11 & 1&1&0&1&1&0&1&1&0&1&1\\
				2&1&0&1&0&1&0&1&0&1&0&1 \\
				3&1&1&1&1&0&1&1&1&1&0&1 \\
				4&1&0&0&0&1&0&1&0&0&0&1 \\
				5&1&1&1&1&1&1&0&1&1&1&1 \\
				6&1&0&1&0&1&0&1&0&1&0&1 \\
				8&1&1&0&1&1&0&1&1&0&1&1 \\
				9&1&0&1&0&0&0&1&0&1&0&1\\
				10&1&1&1&1&1&1&1&1&1&1&0\\
			\end{block}
		\end{blockarray}
		\]
		After necessary calculations, we find that the eigenvalues of the adjacency matrix of $TCG_6$ are $-1, -1, 0, 0, 0, -3.21, -2.29, -1.36, 0.33, 0.67$, and $7.87$. We find that $-1$ occurs as an eigenvalue of the adjacency matrix of $TCG_{11}$ twice which is in accordance with \Cref{number2}.
	\end{example}

	\begin{theorem}
		\label{claim1}
		If $n>3$, then the adjacency matrix of $TCG_n$ is singular.
	\end{theorem}
	
	\begin{proof}
		Since $n\ge 4$, it follows that $2$ will always be a vertex of $TCG_n$.
		Now consider a vertex $v$ of $TCG_n$. Then we claim that $\gcd(v,2)=1$ if and only if $\gcd(v,2^k)=1$ for any natural number $k$. Assume that $\gcd(v,2)=1$. Let $d=gcd(v,2^k)$. This implies $d\mid 2^k$, which further implies $d$ to be an even number. Since $\gcd(v,2)=1$, it is not possible for $d$ must be an even number. Hence we reach a contradiction. The converse part is obvious.
		Thus, we find that $v$ is adjacent to the vertex $2$ of $TCG_n$ if and only if $v$ is adjacent to the vertex $2^k$ of $TCG_n$ for all natural numbers $k$.
		This shows that the row of the adjacency matrix corresponding to the vertex $2$ of $TCG_n$ is the same as the rows corresponding to the vertex $2^k$ of $TCG_n$ for all natural numbers $k$. Thus we find that there always exists two rows in the adjacency matrix of $TCG_n$ which are similar.
		This proves that the determinant of the adjacency matrix of $TCG_n$ is zero. Thus the adjacency matrix of $TCG_n$ is always singular.
	\end{proof}
	
	Now, the question that naturally arises is if it is possible to determine how many times $0$ will occur as an eigenvalue of the adjacency matrix of $TCG_n$.
	
	In the next theorem, we shall provide a result through which we can determine a lower bound on the multiplicity of the eigenvalue $0$ of the adjacency matrix of $TCG_n$.
	
	\begin{theorem}
		\label{number4}
		
		Let \( TCG_n \) be the coprime graph with vertex set \( \{1, 2, \ldots, n\} \), where \( n > 3 \). Let \( \{p_1, p_2, \ldots, p_m\} \) denote the set of all prime numbers less than or equal to \( n \). For each \( p_i \), define \( k_i \) to be the largest positive integer such that \( p_i^{k_i} \leq n \) but \( p_i^{k_i+1} > n \). Then the multiplicity of the eigenvalue \( 0 \) in the adjacency matrix of \( TCG_n \) is at least $
		\left( \sum_{i=1}^{m} k_i \right) - m.$
	\end{theorem}

	\begin{proof}
		Let us consider the prime number $p_1$ such that $p_1< n$.
		Moreover, assume that $p_1^{k_1}\le n$ but $p_1^{k_1+1}>n$.
		Now, following similar arguments as given in \Cref{claim1}, it is clear that $\gcd(p_1,n)=1$ if and only if $\gcd(p_1^{t_1},n)=1$, where $1\le t_1\le k_1$.
		This shows that the row corresponding to the vertex $p_1$ is identical to the rows corresponding to the vertices $p_1^{t_1}$ for $1\le t_1\le k_1$.
		Hence, the adjacency matrix of $TCG_n$ will have $0$ as an eigenvalue with multiplicity at least $k_1-1$.
		
		If we apply the same argument as used for $p_1$ to all other prime numbers $p_i$ where $1\le i\le m$, we find that the adjacency matrix of $TCG_n$ will have $0$ as an eigenvalue with multiplicity at least $k_i-1$ for each $1\le i\le m$. This proves that the adjacency matrix of $TCG_n$ will have $0$ as an eigenvalue with multiplicity at least $ \underbrace{(k_1-1)+(k_2-1)+\cdots+(k_m-1)}_{m \hspace{2mm} \text{times}}$. Thus, the multiplicity of $0$ as an eigenvalue of the adjacency matrix of $TCG_n$ is at least $(\sum_{i=1}^m k_i) -m$.
	\end{proof}
	Now, we shall elaborate \Cref{number4} through an example.
	\begin{example}
		Consider $n=12$.
		Then the prime numbers that are less than $12$ are $2,3,5,7,11$.
		So $m=5$.
		Here $p_1=2,p_2=3,p_3=5,p_4=7,p_5=11$.
		Now, $p_1^3=2^3=8<12.$ So, $k_1=3.$
		Again, $p_2^2=3^2=9<12$. So, $k_2=2$.
		Similarly, $k_3=k_4=k_5=1$.
		Thus, according to \Cref{number4}, the number of times the adjacency matrix of $TCG_n$ will have $0$ as an eigenvalue will be at least $(3+2+1+1+1)-5 = 3.$
		The eigenvalues of the adjacency matrix of  $TCG_{12}$ are $-1, -1, 0, 0, 0, 0, -3.49, -2.59, -1.36, 0.37, 0.95$, and $8.13$.
		So \Cref{number4} provides a lower bound on the multiplicity of the eigenvalue $0$ of the adjacency matrix of $TCG_{12}$.
		
	\end{example}
	
	In the next example, we show that the bound obtained in \Cref{number4} on the multiplicity of the eigenvalue $0$ of the adjacency matrix of $TCG_n$
	is {strict}.
	
	\begin{example}
		Consider $n=10$.
		Then the prime numbers that are less than $10$ are $2,3,5,7$.
		So $m=4$.
		Here $p_1=2,p_2=3,p_3=5,p_4=7$.
		Now, $p_1^3=2^3=8<12.$ So, $k_1=3.$
		Again, $p_2^2=3^2=9<12$. So, $k_2=2$.
		Similarly, $k_3=k_4=1$.
		Thus, according to \Cref{number4}, the number of times the adjacency matrix of $TCG_n$ will have $0$ as an eigenvalue will be at least $(3+2+1+1)-4 = 3.$
		The eigenvalues of the adjacency matrix of  $TCG_{10}$ are $-1, 0, 0, 0, -3.07, -2.21, -1.33, 0.29, 0.67$ and $ 6.65$
		So \Cref{number4} provides a {strict} lower bound on the multiplicity of the eigenvalue $0$ of the adjacency matrix of $TCG_{10}$.
		
	\end{example}

	For large $n$, it is quite difficult to determine all the eigenvalues of the adjacency matrix of a graph. Consequently,  bounds are provided on the spectral radius of $G$, see for example \cite{hong2001sharp, das2004some}.
	
	In the last result, we shall provide a lower bound on the spectral radius of $TCG_n$.
	\begin{theorem}
		If $n\ge 3$, then the spectral radius of $TCG_n$ is bounded below by $2$, where equality holds if and only if $n=3$.
	\end{theorem}
	
	\begin{proof}
		We find that for any $n\ge 3$, the adjacency matrix of $TCG_n$ is a principal submatrix of the adjacency matrix of $TCG_{n+1}$.
		If we  consider $TCG_3$, then the adjacency matrix for $TCG_n$ looks like this:
		\[
		A=
		\left[
		\begin{array}{cccccc}
			0      & 1 & 1  \\
			1     & 0 & 1\\
			1 & 1 & 0
		\end{array}
		\right]
		\]
		The eigenvalues of $A$ are $2$ with multiplicity $1$, and $-1$ with multiplicity $2$. 
		Since the adjacency matrix of a graph is a symmetric matrix, 
		using \Cref{ELT} we observe that for any $n\ge 3$, the largest eigenvalue of the adjacency matrix of $TCG_n$ is greater than the largest eigenvalue of the adjacency matrix of $TCG_3$.
		Thus, the spectral radius of $TCG_n$ is greater than $2$, where equality holds if and only if $n=3$.
	\end{proof}
	
	\section{Conclusion}
	
	In this paper, we have studied the coprime graph of the set of integers.
	Initially, we studied several structural properties of $TCG_n$.
	We then studied the adjacency spectrum of $TCG_n$.
	In this section, we shall provide some directions for future work for the readers. In \Cref{number}, we found a lower bound on the multiplicity of the eigenvalue $-1$ of the adjacency matrix of $TCG_n$. We encourage the readers to find the exact multiplicity of the eigenvalue $-1$.
	Similarly, in \Cref{number4}, we found a lower bound on the multiplicity of the eigenvalue $0$ of the adjacency matrix of $TCG_n$. We encourage the readers to find the exact multiplicity of the eigenvalue $0$.

	\section{Appendix}
	\label{appendix}
	In this section, we list out the eigenvalues of $TCG_n$ for $3\le n\le 15.$
	The calculations have been done using MATLAB.
	\begin{table}[H]
		\centering
		\begin{tabular}{c|c}
			$n$  &  \text{Adjacency Spectrum}\\
			\hline
			3 & [2, -1, -1]
			\\
			\hline
			4 & [0, -1, -1.56, 2.56]
			\\
			\hline
			5 & [0, -1, -1, -1.64, 3.64]
			\\
			\hline
			6 & [0, -1, -2, -1.29, 0.39, 3.89]
			\\
			\hline
			7 & [0, -1, -1, -2.16, -1.29, 0.42, 5.04]
			\\
			\hline
			8 & [0, 0, -1, -1, -2.62, -1.39, 0.44, 5.56]
			\\
			\hline
			9 & [-1, -1, 0, 0, 0, -2.67, -2.14, 0.49, 6.32]
			\\
			\hline
			10 & [-1, 0, 0, 0, -3.07, -2.21, -1.33, 0.29, 0.67, 6.65]\\
			\hline
			11 & [-1, -1, 0, 0, 0, -3.21, -2.29, -1.36, 0.33, 0.67, 7.87]\\
			\hline
			12 & [-1, -1, 0, 0, 0, 0, -3.49, -2.59, -1.36, 0.37, 0.95, 8.13]\\
			\hline
			13 & [-1, -1, -1, 0, 0, 0, 0, -3.70, -2.61, -1.38, 0.39, 0.96, 9.35]\\
			\hline
			14 & [-1, -1, 0, 0, 0, 0, -1.61, 0.61, -4.03, -2.72, -1.25, 0.26, 1.01, 9.72]
			\\
			\hline
			15 & [-1, -1, 0, 0, 0, 0, -4.09, -3.26, -2.02, -1.37, 0.14, 0.49, 0.78, 1.11, 10.22]
		\end{tabular}
		\caption{Eigenvalues of Adjacency Matrix of $TCG_n$}
		\label{tab:my_label}
	\end{table}


\begin{thebibliography}{99}
		
		\bibitem{alsaluli2024laplacian}
		A.~Alsaluli, W.~Fakieh, and H.~Alashwali, ``Laplacian spectrum and vertex connectivity of the unit graph of the ring $\mathbb{Z}_{p^rq^s}$,'' \emph{Axioms}, vol.~13, no.~12, p.~873, 2024.
		
		\bibitem{rehman2024randic}
		N.~U. Rehman, A.~M. Alghamdi, and E.~S. Almotairi, ``Randić spectrum of the weakly zero-divisor graph of the ring $\mathbb{Z}_n$,'' \emph{AKCE International Journal of Graphs and Combinatorics}, pp.~1--8, 2024.
		
		\bibitem{rehman2024exploring}
		N.~U. Rehman and M.~Nazim, ``Exploring normalized distance Laplacian eigenvalues of the zero-divisor graph of ring $\mathbb{Z}_n$,'' \emph{Rendiconti del Circolo Matematico di Palermo Series 2}, vol.~73, no.~2, pp.~515--526, 2024.
		
		\bibitem{madhumitha2024graphs}
		S.~Madhumitha and S.~Naduvath, ``Graphs on groups in terms of the order of elements: A review,'' \emph{Discrete Mathematics, Algorithms \& Applications}, vol.~16, no.~3, 2024.
		
		\bibitem{rather2023normalized}
		B.~A. Rather, H.~A. Ganie, and M.~Aouchiche, ``On normalized distance Laplacian eigenvalues of graphs and applications to graphs defined on groups and rings,'' \emph{Carpathian Journal of Mathematics}, vol.~39, no.~1, pp.~213--230, 2023.
		
		\bibitem{shen2023laplacian}
		S.~Shen, W.~Liu, and W.~Jin, ``Laplacian eigenvalues of the unit graph of the ring $\mathbb{Z}_n$,'' \emph{Applied Mathematics and Computation}, vol.~459, p.~128268, 2023.
		
		\bibitem{banerjee2023distance}
		Banerjee, S. (2023). Distance Laplacian spectra of various graph operations and its application to graphs on algebraic structures. \emph{Journal of Algebra and Its Applications}, 22(01), 2350022.
		
		\bibitem{banerjee2023structural}
		Banerjee, S. (2023). On structural and spectral properties of reduced power graph of finite groups. \emph{Asian-European Journal of Mathematics}, 16(09), 2350170.
		
		\bibitem{sriramgeneralised}
		S.~Sriram, P.~Veeramallan, and A.~D. Christopher, ``A generalised coprime graph—revisited,'' \emph{International Journal of Applied Graph Theory}, vol.~7, no.~1, pp.~1--10, 2023.
		
		\bibitem{banerjee2022spectra}
		Banerjee, S. (2022). Spectra and topological indices of comaximal graph of $\mathbb Z_n$. \emph{Results in Mathematics}, 77(3), 111.
		
		\bibitem{banerjee2022laplacian}
		Banerjee, S. (2022). Laplacian spectrum of comaximal graph of the ring $\mathbb Z_n$. \emph{Special Matrices}, 10(1), 285--298.
		
		\bibitem{hamm2021parameters}
		J.~Hamm and A.~Way, ``Parameters of the coprime graph of a group,'' \emph{International Journal of Group Theory}, vol.~10, no.~3, pp.~137--147, 2021.
		
		\bibitem{banerjee2021laplacian}
		S.~Banerjee, ``Laplacian spectra of coprime graph of finite cyclic and dihedral groups,'' \emph{Discrete Mathematics, Algorithms and Applications}, vol.~13, no.~3, p.~2150020, 2021.
		\bibitem{ma2014coprime}
		X.~L. Ma, H.~Q. Wei, and L.~Y. Yang, ``The coprime graph of a group,'' \emph{International Journal of Group Theory}, vol.~3, no.~3, pp.~13--23, 2014.
		
		\bibitem{brouwer2011spectra}
		A.~E. Brouwer and W.~H. Haemers, \emph{Spectra of Graphs}, Springer, 2011.
		
		\bibitem{rao2011creative}
		S.~N. Rao, ``A creative review on coprime (prime) graphs,'' \emph{Lecture Notes, DST Workshop, WGTA, BHU, Varanasi}, pp.~1--24, 2011.
		
		\bibitem{sander2009kernel}
		J.~W. Sander and T.~Sander, ``On the kernel of the coprime graph of integers,'' \emph{Integers}, vol.~9, no.~5, pp.~569--579, 2009.
		
		\bibitem{ahlswede2006maximal}
		R.~Ahlswede and V.~Blinovsky, ``Maximal sets of numbers not containing $k+1$ pairwise coprimes and having divisors from a specified set of primes,'' \emph{Journal of Combinatorial Theory, Series A}, vol.~113, no.~8, pp.~1621--1628, 2006.
		
		\bibitem{das2004some}
		K.~C. Das and P.~Kumar, ``Some new bounds on the spectral radius of graphs,'' \emph{Discrete Mathematics}, vol.~281, no.~1--3, pp.~149--161, 2004.
		
		\bibitem{west2001introduction}
		D.~B. West, \emph{Introduction to Graph Theory}, 2nd ed., Prentice Hall, Upper Saddle River, 2001.
		
		\bibitem{hong2001sharp}
		Y.~Hong, J.~L. Shu, and K.~Fang, ``A sharp upper bound of the spectral radius of graphs,'' \emph{Journal of Combinatorial Theory, Series B}, vol.~81, no.~2, pp.~177--183, 2001.
		
		\bibitem{pan2019full}
		J.~Pan and X.~Guo, ``The full automorphism groups, determining sets and resolving sets of coprime graphs,'' \emph{Graphs and Combinatorics}, vol.~35, no.~2, pp.~485--501, 2019.
		
		\bibitem{selvakumar2017classification}
		K.~Selvakumar and M.~Subajini, ``Classification of groups with toroidal coprime graphs,'' \emph{Australasian Journal of Combinatorics}, vol.~69, no.~2, pp.~174--183, 2017.
		
		\bibitem{dorbidi2016note}
		H.~R. Dorbidi, ``A note on the coprime graph of a group,'' \emph{International Journal of Group Theory}, vol.~5, no.~4, pp.~17--22, 2016.
		
		\bibitem{ireland2013classical}
		K.~Ireland and M.~Rosen, \emph{A Classical Introduction to Modern Number Theory}, vol.~84, Springer, 2013.
		
		\bibitem{schaefer2012graph}
		M.~Schaefer, ``The graph crossing number and its variants: A survey,'' \emph{Electronic Journal of Combinatorics}, Article DS21, 2012.
		
		\bibitem{sriram2014}
		S.~Mutharasu, N.~Mohamed Rilwan, M.~K. Angel Jebitha, and T.~Tamizh Chelvam, ``On generalized coprime graphs,'' \emph{Iranian Journal of Mathematical Sciences and Informatics}, vol.~9, no.~2, pp.~1--6, 2014.
		
		\bibitem{erdHos1997cycles}
		P.~Erd\H{o}s and G.~N. S\'ark\"ozy, ``On cycles in the coprime graph of integers,'' \emph{Electronic Journal of Combinatorics}, Article R8, 1997.
		
		\bibitem{ahlswede1996sets}
		R.~Ahlswede and L.~H. Khachatrian, ``Sets of integers and quasi-integers with pairwise common divisor,'' \emph{Acta Arithmetica}, vol.~74, no.~2, pp.~141--153, 1996.
		
		\bibitem{ahlswede1995maximal}
		R.~Ahlswede and L.~H. Khachatrian, ``Maximal sets of numbers not containing $k+1$ pairwise coprime integers,'' \emph{Acta Arithmetica}, vol.~72, no.~1, pp.~77--100, 1995.
		
		\bibitem{ahlswede1994extremal}
		R.~Ahlswede and L.~H. Khachatrian, ``On extremal sets without coprimes,'' \emph{Acta Arithmetica}, vol.~66, no.~1, pp.~89--99, 1994.
		
		\bibitem{erdos1973crossing}
		P.~Erd\H{o}s and R.~K. Guy, ``Crossing number problems,'' \emph{American Mathematical Monthly}, vol.~80, no.~1, pp.~52--58, 1973.
		
		\bibitem{garey1983crossing}
		M.~R. Garey and D.~S. Johnson, ``Crossing number is NP-complete,'' \emph{SIAM Journal on Algebraic Discrete Methods}, vol.~4, no.~3, pp.~312--316, 1983.
		
		\bibitem{pomerance1980proof}
		C.~Pomerance and J.~L. Selfridge, ``Proof of D.~J. Newman's coprime mapping conjecture,'' \emph{Mathematika}, vol.~27, no.~1, pp.~69--83, 1980.
		
		\bibitem{erdos1961remarks}
		P.~Erd\H{o}s, ``Remarks on number theory. I,'' \emph{Mat. Lapok}, vol.~12, pp.~10--17, 1961.
		
	\end{thebibliography}
\end{document}